\documentclass[10pt]{article}
\usepackage{amsmath,amssymb}
\usepackage[mathscr]{eucal}

\usepackage{mathtools}                                                      
\mathtoolsset{showonlyrefs=true}                                    

\usepackage{eufrak}                                             

\def \cost {M}

\def \Kol {\mathcal{K}_{\cost,B}}

\usepackage{amsmath}                                                        
\usepackage{graphicx}                                                       
\usepackage{subfigure}
\usepackage{enumitem}                                                    
\usepackage{hyperref}
\usepackage{amssymb}                                                        
\usepackage[mathscr]{eucal}                                             
\usepackage{cancel}                                                             
\usepackage[normalem]{ulem}                                                                 
\usepackage{pstricks}
\usepackage{rotating}
\usepackage{lscape}
\usepackage[paperwidth=8.5in,paperheight=11in,top=1.25in, bottom=1.25in, left=1.00in, right=1.00in]{geometry}
\usepackage{mathtools}                                                      
\mathtoolsset{showonlyrefs=true}                                    
\usepackage{fixltx2e,amsmath}                                           
\MakeRobust{\eqref}
\usepackage{mathdots}
\usepackage{amsthm}                                                             
\allowdisplaybreaks                                                             
\usepackage{chngcntr}

\usepackage{titlesec}
\usepackage[titletoc,toc,title]{appendix}

\usepackage{apptools, etoolbox}

\AtAppendix{\counterwithin{theorem}{section}}

\usepackage{color}

\newtheorem{theorem}{Theorem}[section]

\newtheorem{corollary}[theorem]{Corollary}
\newtheorem{lemma}[theorem]{Lemma}
\newtheorem{definition}[theorem]{Definition}

\newtheorem{remark}[theorem]{Remark}
\newtheorem{assumption}[theorem]{Assumption}
\newtheorem{notation}[theorem]{Notation}

\def \k {{\kappa}}

\def \a {{\alpha}}
\def \b {{\beta}}
\def \d {{\delta}}
\def \l {{\lambda}}

\def \G {{\Gamma}}
\def \s {{\sigma}}

\def \R {{\mathbb {R}}}

\def \x {{\xi}}
\def \e {{\varepsilon}}

\def \r {{\varrho}}

\def \t {{\tau}}
\def \t {{\tau}}
\def \n {{\nu}}
\def \m {{\mu}}
\def \y {{\eta}}

\def \z {{\zeta}}

\def \LL {{L_{0}}}

\def \Gg {{\G_{0}}}

\def \g {{\gamma}}
\def \O {{\Omega}}
\def \phi {{\varphi}}
\def \div {{\text{\rm div}}}

\def \tilde {\widetilde}

\def\p{\partial}

\def \k {{\kappa}}

\def \a {{\alpha}}
\def \b {{\beta}}
\def \d {{\delta}}
\def \l {{\lambda}}

\def \G {{\Gamma}}
\def \s {{\sigma}}

\def \Dil {\mathcal{D}}

\def \R  {{\mathbb {R}}}
\def \x {{\xi}}
\def \g {{\gamma}}
\def \e {{\varepsilon}}

\def \t {{\tau}}
\def \n {{\nu}}
\def \m {{\mu}}
\def \y {{\eta}}

\def \z {{\zeta}}
\def \p {{\partial}}
\def \a {{\alpha}}
\def \O {{\Omega}}

\def \d {{\delta}}

\def \k {{\kappa}}

\def \a {{\alpha}}
\def \b {{\beta}}
\def \d {{\delta}}

\def \G {\Ga}

\def \Ga {{\Gamma}}

\def \s {{\sigma}}

\def \R {{\mathbb {R}}}

\def \x {{\xi}}
\def \e {{\varepsilon}}

\def \r {{\varrho}}

\def \t {{\tau}}
\def \t {{\tau}}
\def \n {{\nu}}

\def \m {{\mu}}
\def \y {{\eta}}

\def \z {{\zeta}}

\def \g {{\gamma}}
\def \O {{\Omega}}
\def \phi {{\varphi}}
\def \div {{\text{\rm div}}}

\def \tilde {\widetilde}

\def\l {\lambda}

\def \à {{\`a }}
\def \è {{\`e }}
\def \ò {{\`o }}
\def \ù {{\`u }}

\makeatletter
\def\section{\@startsection {section}{1}{\z@}{3.25ex plus 1ex minus
 .2ex}{1.5ex plus .2ex}{\large\bf}}
\def\subsection{\@startsection{subsection}{2}{\z@}{3.25ex plus 1ex minus
 .2ex}{1.5ex plus .2ex}{\normalsize\bf}}
\@addtoreset{equation}{section} 
\makeatother

\title{\empty}

\author{\empty}

\date{\empty}

\numberwithin{equation}{section}

\begin{document}

\title{Nash estimates and upper bounds for non-homogeneous Kolmogorov equations}

\author{Alberto Lanconelli\thanks{Dipartimento di Matematica, Universit\`a di Bari Aldo Moro, Bari, Italy. \textbf{e-mail}: alberto.lanconelli@uniba.it} \and Andrea Pascucci\thanks{Dipartimento di Matematica,
Universit\`a di Bologna, Bologna, Italy. \textbf{e-mail}: andrea.pascucci@unibo.it}}

\date{This version: \today}

\maketitle

\begin{abstract}
We prove a Gaussian upper bound for the fundamental solutions of a class of ultra-parabolic
equations in divergence form. The bound is independent on the smoothness of the coefficients and
generalizes some classical results by Nash, Aronson and Davies. The class considered has relevant
applications in the theory of stochastic processes, in physics and in mathematical finance.
\end{abstract}


%
%

\section{Introduction}
We consider the Kolmogorov-type equation with measurable coefficients
\begin{equation}\label{PDE}
 Lu:=\sum_{i,j=1}^{m_0}\partial_{x_i}(a_{ij}\partial_{x_j}u)+\sum_{i=1}^{m_0}\partial_{x_i}(a_{i} u)+
 c u+\sum_{i,j=1}^{d}b_{ij}x_j\partial_{x_i}u+\partial_tu=0,\qquad (t,x)\in\R\times\R^{d}.
\end{equation}
where $m_0\leq d$ and $L$ verifies the following two standing assumptions:
\begin{assumption}\label{assA}
The coefficients $a_{ij}=a_{ji},a_i,c$, for $1\le i,j\le m_0$, are bounded, measurable functions
such that
\begin{equation}\label{ellipticity}
 \m^{-1}|\x|^{2}\le \sum_{i,j=1}^{m_{0}}a_{ij}(t,x)\x_{i}\x_{j}\le \m|\x|^{2},\qquad
 \x\in\R^{m_{0}},\ (t,x)\in\R^{d+1},
\end{equation}
for some positive constant $\m$.
\end{assumption}
\begin{assumption}\label{assB}
The matrix $B:=\left(b_{ij}\right)_{1\leq i,j\leq d}$ has constant real entries and takes the
block-form
\begin{equation}\label{e65b}
  B=\begin{pmatrix}
 \ast & \ast & \cdots & \ast & \ast \\ B_1 & \ast &\cdots& \ast & \ast \\ 0 & B_2 &\cdots& \ast& \ast \\ \vdots & \vdots
 &\ddots& \vdots&\vdots \\ 0 & 0 &\cdots& B_{\nu}& \ast
  \end{pmatrix}
\end{equation}
where each $B_i$ is a $\left(m_{i}\times m_{i-1}\right)$-matrix of rank $m_{i}$ with
\begin{equation}
 m_0\geq m_1\geq \cdots \geq m_{\nu}\geq 1, \qquad \sum_{i=0}^{\nu} m_i = d,
\end{equation}
and the blocks denoted by {\rm ``$\ast$''} are arbitrary.
\end{assumption}

Degenerate equations of the form \eqref{PDE} naturally arise in the theory of stochastic
processes, in physics and in mathematical finance. For instance, if $W$ denotes a real Brownian
motion, then the simplest non-trivial Kolmogorov operator
  $$\frac{1}{2}\p_{vv}+v\p_{x}+\p_{t},\qquad t\ge 0,\,(v,x)\in\R^{2},$$
is the infinitesimal generator of the classical Langevin's 
stochastic equation
  $$
  \begin{cases}
    dV_{t}=dW_{t}, \\
    dX_{t}=V_{t}dt,
  \end{cases}
  $$
that describes the position $X$ and velocity $V$ of a particle in the phase space (cf.
\cite{Langevin}). Notice that in this case we have $1=m_{0}<d=2$.

Linear Fokker-Planck equations (cf. \cite{Desvillettes} and \cite{Risken}), non-linear
Boltzmann-Landau equations (cf. \cite{Lions1} and \cite{Cercignani}) and non-linear equations for
Lagrangian stochastic models commonly used in the simulation of turbulent flows (cf. \cite{Talay})
can be written in the form
\begin{equation}\label{PDE1}
 \sum_{i,j=1}^{n}\partial_{v_i}(a_{ij}\partial_{v_j}f)+\sum_{j=1}^{n}v_{j}\p_{x_{j}}f+\p_{t}f=0,\qquad
 t\ge 0,\, v\in\R^{n},\, x\in\R^{n},
\end{equation}
with the coefficients $a_{ij}=a_{ij}(t,v,x,f)$ that may depend on the solution $f$ through some
integral expressions. Clearly \eqref{PDE1} is a particular case of \eqref{PDE} with $n=m_{0}<d=2n$
and
  $$B=\begin{pmatrix}
    0 & 0 \\
    I_{n} & 0 \
  \end{pmatrix}$$
where $I_{n}$ denotes the $\left(n\times n\right)$-identity matrix.

In mathematical finance, equations of the form \eqref{PDE} appear in various models for the
pricing of path-dependent derivatives such as Asian options (cf., for instance,
\cite{pascuccibook}, \cite{BarucciPolidoroVespri}), stochastic volatility models (cf.
\cite{HobsonRogers}, \cite{Peszek}) and in the theory of stochastic utility (cf.
\cite{AntonelliBarucciMancino}, \cite{AntonelliPascucci}). In interest rate modeling, equations of
the type \eqref{PDE} were used in the study of the possible realization of Heath-Jarrow-Morton
models in terms of a finite dimensional Markov diffusion (cf. \cite{Ritchken},
\cite{ChiarellaKwon}).

A systematic study of Kolmogorov operators has been carried out by several authors. In the case of
constant coefficients, Kupcov \cite{Kupcov2}, Lanconelli and Polidoro \cite{LanconelliPolidoro}
studied the geometrical properties of the operator, giving necessary and sufficient conditions for
the existence of the fundamental solution. In the case of {\it H\"older continuous coefficients}
and assuming invariance properties with respect to a suitable homogeneous Lie group, existence of
a fundamental solution has been proved by Weber \cite{Weber}, Il'in \cite{Il'in}, Eidelman
\cite{Eidelman} and Polidoro \cite{Polidoro2}; pointwise upper and lower bounds for the
fundamental solution, mean value formulas and Harnack inequalities are given in \cite{Polidoro2}
and \cite{Polidoro1}. Schauder type estimates have been proved by Satyro \cite{Satyro}, Lunardi
\cite{Lunardi}, Manfredini \cite{Manfredini}. In the more general case of non-homogeneous
Kolmogorov equations with H\"older continuous coefficients, the existence of a fundamental
solution has been proved by Morbidelli \cite{Morbidelli} and Di Francesco and Pascucci
\cite{DiFrancescoPascucci2}; Harnack inequalities and Schauder estimates were proved by Di
Francesco and Polidoro \cite{DiFrancescoPolidoro}. The first results for Kolmogorov operators with
{\it measurable coefficients} were proved by Cinti, Pascucci and Polidoro in \cite{PP2003},
\cite{CPP2008}.

\medskip
The main result of this paper is a Gaussian upper bound, independent of the smoothness of the
coefficients, for the transition density/fundamental solution $\G=\G(t,x;T,y)$ of \eqref{PDE}.
Before stating our result, we introduce the following
\begin{notation}\label{notation1}
Let $\cost>0$ and $B:=\left(b_{ij}\right)_{1\leq i,j\leq d}$ a matrix that satisfies Assumption
\ref{assB}. We denote by $\Kol$ the class of Kolmogorov operators $L$ of the form \eqref{PDE},
that satisfy Assumption \ref{assA} with the non-degeneracy constant $\m$ in \eqref{ellipticity}
and the norms
$\|a_{i}\|_{\infty}$, $\|c\|_{\infty}$ smaller than $\cost$. 
\end{notation}
 The following theorem generalizes the classical results by Nash
\cite{Nash},\cite{Fabes1993}, Aronson \cite{Aronson} and Davies \cite{Davies} for uniformly
parabolic equations and provides step forward for the study of non-linear Kolmogorov equations.
\begin{theorem}\label{t1}
Let $L\in \Kol$ and $T_{0}>0$. There exists a positive constant $C$, only dependent on $\cost,B$ and $T_{0}$, 
such that
\begin{align}\label{thes}
 \G(t,x;T,y)\le \frac{C}{(T-t)^{\frac{Q}{2}}}\exp\left(-\frac{1}{C}\left|\Dil\left(\left(T-t\right)^{-\frac{1}{2}}\right)
 \left(x-e^{-(T-t)B}y\right)\right|^2\right),
\end{align}
for $0<T-t\le T_{0}$ and $x,y\in\R^{d}$, with
\begin{equation}\label{e15}
 {\Dil(r) :=\text{\rm diag}(r I_{m_{0}},r^{3} I_{m_{1}},\dots,r^{2{\nu}+1}I_{m_{{\nu}}}),\qquad r>0,}
\end{equation}
where $I_{m_{i}}$ denotes the $\left(m_{i}\times m_{i}\right)$-identity matrix, and
\begin{equation}\label{e77}
 Q:=m_{0}+3m_{1}+\cdots+(2\nu+1)m_{\nu}.
\end{equation}
\end{theorem}
The exponent $\frac{Q}{2}$ appearing in estimate \eqref{thes} is optimal, as it can be easily seen
in the case of constant-coefficient Kolmogorov operators whose fundamental solution is explicit
(see \eqref{e22and}). Notice the difference with respect to the uniformly parabolic case: for
instance in $\R^{3}$, for the heat operator $\p_{xx}+\p_{yy}+\p_{t}$ we have $Q=2$, while for the
prototype Kolmogorov operator $\p_{xx}+x\p_{y}+\p_{t}$ we have $Q=4$. To explain the specific form
of the Gaussian exponent in \eqref{thes} and the role of the constant $Q$ in \eqref{e77} (that is
typically greater than the Euclidean dimension $d$), we recall some basic facts about Kolmogorov
operators and the Lie group structures on $\R^{d+1}$ naturally associated with them.

Let us first rewrite \eqref{PDE} in the compact form
\begin{equation}\label{L}
 Lu=\div(AD  u)+\div(au)+cu+Yu=0,
\end{equation}
{where $D=(\p_{x_{1}},\dots,\p_{x_{d}})$ denotes the gradient in $\R^{d}$,}
$A:=\left(a_{ij}\right)_{1\leq i,j\leq d}$, $a:=\left(a_i\right)_{1\leq i\leq d}$ with
$a_{ij}=a_{i}\equiv 0$ for $i>m_0$ or $j>m_0$, and
\begin{equation}\label{e14c}
 Y:=
 \langle B x, D   \rangle + \partial_t,\qquad (t,x)\in\R\times\R^{d}.
\end{equation}
The {\it constant-coefficient Kolmogorov operator}
\begin{equation}\label{e14b}
 \LL:= \frac{1}{2}\sum_{i=1}^{m_{0}}\p_{x_{i}x_{i}}+ Y
\end{equation}
will be referred to as the \emph{principal part of $L$}. It is known that Assumption \ref{assB} is
equivalent to the hypoellipticity of $\LL$: in fact, Assumption \ref{assB} is also equivalent to
the well-known H\"ormander's condition, which in our setting reads:
\begin{align}
 {\rm rank\ Lie}\left(\p_{x_{1}},\dots,\p_{x_{m_{0}}},Y\right)(t,x)=d+1,\qquad\text{for all }(t,x)\in \R^{d+1},
\end{align}
where Lie$\left(\p_{x_{1}},\dots,\p_{x_{m_{0}}},Y\right)$ denotes the Lie algebra generated by the
vector fields $\p_{x_{1}},\dots,\p_{x_{m_{0}}}$ and $Y$ (see Proposition 2.1 in
\cite{LanconelliPolidoro}). Thus operator $L$ can be regarded as a perturbation of its principal
part $L_{0}$: roughly speaking, Assumption \ref{assA} ensures that the sub-elliptic structure of
$L_{0}$ is preserved under perturbation.

Constant-coefficient Kolmogorov operators are naturally associated to {\it linear} stochastic
differential equations: indeed, $\LL$ is the infinitesimal generator of the $d$-dimensional SDE
\begin{equation}\label{SDE}
  dX_{t}=B X_{t}dt+\s dW_{t},
\end{equation}
where $W$ is a standard $m_{0}$-dimensional Brownian motion and $\s$ is the
$(d\times m_{0})$-matrix 
\begin{equation}\label{sigm}
 \s=\begin{pmatrix}
   I_{m_{0}} \\
   0 \
 \end{pmatrix}.
\end{equation}
The solution $X$ of \eqref{SDE} is a Gaussian process with transition density
\begin{align} \label{e22and}
 \Gg(t,x;T,y)
 &=  \frac{1}{  \sqrt{(2\pi)^{d}\det \mathcal{C}(T-t)} }
    \exp\left(-\frac{1}{2}\langle\mathcal{C}(T-t)^{-1} \big(y - e^{(T-t)B}x)\big),
    \big(y -e^{(T-t)B}x\big)\rangle\right)
\end{align}
for  $t<T$ and $x,y\in\R^{d}$, where
\begin{align}
\mathcal{C}(t)=\int\limits_{0}^{t}\left(e^{s B}\s\right)\left(e^{s B}\s\right)^{\ast}ds
\end{align}
is the covariance matrix of $X_{t}$. Assumption \ref{assB} ensures (actually, is equivalent to the
fact) that $\mathcal{C}(t)$ is positive definite for any positive $t$. Moreover, $\Gg$ in
\eqref{e22and} is the fundamental solution of $\LL$ and the function
\begin{align}
 u(t,x):=E\left[\phi\left(X_{T}\right)\mid X_{t}=x\right]=\int_{\R^{d}}\Gg(t,x;T,y)\phi(y)dy,\qquad t<T,\ x\in\R^{d},
\end{align}
solves the backward Cauchy problem
\begin{equation}\label{PC}
  \begin{cases}
    \LL u(t,x)=0,\qquad  & t<T,\ x\in\R^{d}, \\
    u(T,x)=\phi(x) & x\in\R^{d},
  \end{cases}
\end{equation}
for any bounded and continuous function $\phi$.

Operator $\LL$ has some remarkable invariance properties that were first studied in
\cite{LanconelliPolidoro}. Denote by $\ell_{(\t,\x)}$, for $(\t,\x)\in\R^{d+1}$, the
left-translations in $\R^{d+1}$ defined as
\begin{equation}\label{eq:translation}
 \ell_{(\t,\x)}(t,x):=(\t,\x)\circ (t,x):=\left(t+\t,x+{e^{t B}}\x\right),
\end{equation}
Then, $\LL$ is invariant with respect to $\ell_{\z}$ in the sense that
 $$L_{0}\left(u\circ \ell_{\z}\right)=\left(L_{0}u\right)\circ \ell_{\z},\qquad \z\in\R^{d+1}.$$
Moreover, let $\Dil(r)$ be as in \eqref{e15}: then $\LL$ is homogeneous with respect to the
dilations in $\R^{d+1}$ defined as
\begin{equation}\label{e066}
 \d_{r}(t,x):=\left(r^2t,\Dil(r)x\right), \qquad r>0,
\end{equation}
{\it if and only if} all the $\ast$-blocks of $B$ in \eqref{e65b} are null
(\cite{LanconelliPolidoro}, Proposition 2.2): in that case, we have
\begin{equation}\label{e200}
 \LL(u\circ\delta_r)=r^{2}(\LL u)\circ\delta_r,\qquad r>0.
\end{equation}
Since the Jacobian $J \Dil(r)$ equals $r^{Q}$, the natural number $Q$ in \eqref{e77} is usually
called the {\it homogeneous dimension} of $\mathbb{R}^{d}$ with respect to $(\Dil(r))_{r>0}$.

\medskip Now let us consider the case of a Kolmogorov operator $L\in\Kol$ with variable coefficients.
It turns out that the invariance properties of the principal part $L_{0}$ are inherited by $L$ in
terms of ``invariance within the class $\Kol$''. More explicitly, for the left-translations we
have
\begin{remark}\label{rtrasl}
Let $\z\in\R^{d+1}$ and $L\in\Kol$. If $u$ is a solution of
  $Lu=0$
then $v:=u\circ \ell_{\z}$ solves
  $L^{(\z)}v=0$
where $L^{(\z)}$ is obtained from $L$ by left-translating its coefficients, that is
$L^{(\z)}=L\circ \ell_{\z}$. Moreover, operator $L^{(\z)}$ still belongs to $\Kol$.
\end{remark}
As for dilations, we have to distinguish between {\it homogeneous Kolmogorov operators} (i.e.
operators with null $\ast$-blocks in \eqref{e65b}) and general Kolmogorov operators.
\begin{remark}\label{rdil}
Let $\l>0$ and $L\in\Kol$ be a homogeneous Kolmogorov operators. If $u$ is a solution of
  $Lu=0$
then $v:=u\circ\d_{\l}$ solves
  $L^{\l}v=0$
where $L^{\l}$ is obtained from $L$ by dilating its coefficients, that is $L^{\l}=L\circ \d_{\l}$.
Moreover, operator $L^{\l}$ still belongs to $\Kol$.
\end{remark}
By virtue of the previous remarks, it turns out that
the crucial step to achieve estimate
\eqref{thes} is to prove it for $t=0$, $T=1$ and $y=0$, that is
\begin{align}\label{thes1}
 \G(0,x;1,0)\le C\exp\left(-\frac{|x|^2}{C}\right),\qquad \ x\in\R^{d},
\end{align}
with {\it $C$ dependent only on $\cost$ and $B$}. Indeed, once \eqref{thes1} is proved, then the
general estimate for $L\in\Kol$ follows from the invariance of the class $\Kol$ with
respect to the left-translations $\ell$ and the intrinsic dilations $\d$. 
It is then clear that the factor $(T-t)^{-\frac{Q}{2}}$ and the term $\left(x-e^{-(T-t)B}y\right)$
in the exponential of \eqref{theshom} come directly from the use of translations and dilations in
the intrinsic Lie group structure.

Estimate \eqref{thes} is consistent with the following Gaussian upper bound for Kolmogorov
operators with {\it H\"older continuous coefficients}, proved in \cite{Polidoro2} and
\cite{DiFrancescoPascucci2} (see also \cite{Menozzi}, \cite{BallyKohatsu}) by means of the
classical parametrix method:
\begin{align}\label{theshom}
 \G(t,x;T,y)\le C\G_{0}(t,x;T,y),\qquad t<T,\, x\in\R^{d},
\end{align}
with $C=C(M)$ and $\G_{0}$ as in \eqref{e22and} with $\s=\begin{pmatrix}
   \sqrt{2M} I_{m_{0}} \\
   0 \
 \end{pmatrix}.
$ Notice that for {\it homogeneous} Kolmogorov operators, the constant $C$ in estimate
\eqref{thes} is independent of $T-t$. 

\medskip In the case of {\it non-homogeneous} Kolmogorov operators, which is the main focus of this
paper, the lack of homogeneity makes the proof of \eqref{thes1} rather involved. The invariance
property of Remark \ref{rtrasl} remains unchanged, while the scaling argument cannot be used
anymore. However, we have the following result (see Remark 3.2 in \cite{LanconelliPolidoro}).
\begin{lemma}\label{l3}
Let $\l>0$ and $L\in\Kol$. If $u$ is a solution of
  $Lu=0$
then $v:=u\circ\d_{\l}$ solves
  $L^{\l}v=0$
where
\begin{equation}\label{Lr}
  L^{\l}u:=\div(A^{(\l)}D  u)+\langle B^{(\l)}x,D  u\rangle+\partial_t
 u+\div(a^{(\l)}u)+c^{(\l)}u(t,x),
\end{equation}
with
  $$A^{(\l)}(t,x)=A\left(\d_{\l}(t,x)\right),\qquad a^{(\l)}(t,x)=\l a\left(\d_{\l}(t,x)\right),\qquad c^{(\l)}(t,x)=\l^{2}
  c\left(\d_{\l}(t,x)\right),$$
and $B^{(\l)}=\l^{2}\Dil_{\l} B\Dil_{\frac{1}{\l}}$, that is
\begin{equation}\label{e65bl}
  B^{(\l)}=\begin{pmatrix}
 \l^{2} B_{1,1} & \l^{4} B_{1,2} & \cdots & \l^{2\nu} B_{1,\nu} & \l^{2\nu+2} B_{1,\nu+1} \\
 B_1 & \l^{2} B_{2,2} &\cdots& \l^{2\nu-2} B_{2,\nu} & \l^{2\nu} B_{2,\nu+1} \\
 0 & B_2 &\cdots& \l^{2\nu-4} B_{3,\nu}& \l^{2\nu-2} B_{3,\nu+1} \\ \vdots & \vdots
 &\ddots& \vdots&\vdots \\ 0 & 0 &\cdots& B_{\nu}& \l^{2} B_{\nu+1,\nu+1}
  \end{pmatrix},
\end{equation}
where $B_{i,j}$ denotes the $\ast$-block in the $(i,j)$-th position in \eqref{e65b}.
\end{lemma}
More importantly, {\it we will show that if $L\in\Kol$ then the fundamental solution $\G^{\l}$ of
$L^{\l}$ in \eqref{Lr} satisfies estimate \eqref{thes1} uniformly with respect to $\l\in[0,1]$,
that is with the constant $C$ dependent only on $\cost$ and $B$:} intuitively, this is due to the
fact that, on the one hand, the dilations $\d_{\l}$ do not affect the blocks $B_{1},\dots,B_{\n}$
in \eqref{e65bl} (this guarantees the hypoellipticity of the operator, uniformly with respect to
$\l$); on the other hand, the ``new'' $\ast$-blocks in \eqref{e65bl} are bounded functions of
$\l\in[0,1]$.

The first step in the proof of \eqref{thes1} 
consists in proving the local boundedness of non-negative weak solutions of equation \eqref{PDE}
(cf. Theorem \ref{moser}): this result slightly extends the Moser's estimates obtained in Cinti et
al. \cite{CPP2008} where the lower order terms were not included. From the Moser's estimates we
get a Nash upper bound for the fundamental solution in Theorem \ref{t4}. Next we employ a method
by Aronson \cite{Aronson} which provides the crucial estimates in Theorem \ref{t2} and the desired
Gaussian upper bound: here we extend the previous scaling argument to non-homogeneous
Kolmogorov operators and then prove Theorem \ref{t1} for the general class $\Kol$. 


\section{Moser's iterative method}
In this section we adapt the Moser's iterative method to prove the local boundedness of weak
solutions of $L$. In the classical setting, Moser's approach combines Caccioppoli type estimates
with the embedding Sobolev inequality: for Kolmogorov operators that are not uniformly parabolic,
Caccioppoli estimates provide $L^{2}_{\text{\rm loc}}$-bounds only for the first $m_{0}$
derivatives (cf. Assumption \ref{assA}) and therefore a naive extension of the classical approach
is not possible. Nevertheless, this problem can be overcome by using the original argument
proposed in \cite{PP2004}, that is based on some ad hoc Sobolev type inequalities for local
solutions to \eqref{PDE}.

To introduce the main result of this section, Theorem \ref{moser} below, we recall the definition
of weak solution.
\begin{definition}
We say that $u$ is a \emph{weak sub-solution} of \eqref{PDE} in a domain $\O$ of $\R^{d+1}$ if
$$u,\p_{x_{1}}u,\dots, \p_{x_{m_{0}}}u, Yu\in L^{2}_{\text{\rm loc}}(\Omega)$$ and for any
non-negative $\varphi\in C_0^{\infty}(\Omega)$ we have
\begin{equation}
 \int_{\Omega}-\langle AD  u,D  \phi\rangle-\langle a,D  \phi\rangle u+\varphi cu+\varphi Yu\geq0.
\end{equation}
A function $u$ is a \emph{weak super-solution} if $-u$ is a weak sub-solution. If $u$ is a weak
sub and super-solution, then we say that $u$ is a \emph{weak solution}.
\end{definition}
In the following statement, $R_{r}(z_{0})$ denotes the cylinder
\begin{equation}\label{e76}
 R_{r}(z_{0}):= z_0 \circ \d_{r} \left(R_{1}\right)=\{z\in\R^{d+1}\mid z= z_{0} \circ \d_{r}(\z), \, \z \in
 R_{1}\},\qquad z_{0}\in\R^{d+1},\ r>0,
\end{equation}
where
\begin{align}
 R_{1} = \{(t,x)\in\R\times\R^{d}\mid  |t|<1,\, |x|<1\}.
\end{align}
\begin{theorem}\label{moser}
Let $L\in\Kol$, $\l\in[0,1]$ and $L^{\l}$ as in \eqref{Lr}. Let $u$ be a non-negative weak
solution of $L^{\l}u=0$ in a domain $\O$. Let $z_{0}\in \O$ and $0<\r<r\le r_{0}$ be such that
$r-\r<1$ and $\overline{R_{r}(z_{0})}\subseteq\O$. Then, for every $p>0$ there exists a positive
constant $C=C(\cost,r_{0},p)$ such that
\begin{align}\label{e50}
 \sup_{R_{\r}(z_{0})}u^{p}\le\frac{C}{(r-\r)^{Q+2}}\int\limits_{R_{r}(z_{0})}u^{p}.
\end{align}
Estimate \eqref{e50} also holds for every $p<0$ such that $u^{p}\in L^{1}(R_{r}(z_{0}))$.
\end{theorem}

The proof of Theorem \ref{moser} requires two auxiliary results of independent interest. In the
following statement, we use the notation $D_{m_0}=\left(\p_{x_{1}},\dots,\p_{x_{m_{0}}}\right)$
and {$\|B\|$ for the norm of $B$ as a linear operator}.
\begin{theorem}[Caccioppoli type inequality]\label{Caccioppoli}
Let $L\in\Kol$ and $u$ be a non-negative weak sub-solution of \eqref{PDE} in $R_r(z_0)$, with
$0<\r<r\le r_{0}$ such that $r-\r<1$. If $u^q\in L^{2}(R_r(z_0))$ for some $q>\frac{1}{2}$, then
$D _{m_0}u^q\in L^{2}(R_{\rho}(z_0))$ and there exists a constant $C=C(M,\|B\|)$ such that
\begin{equation}\label{caccioppoli}
 \int_{R_{\rho}(z_0)}|D _{m_0}u^q|^2\leq C\left(\frac{q}{2q-1}\right)^2\frac{q}{(r-\rho)^2}\int_{R_r(z_0)}|u^q|^2.
\end{equation}
If $u$ is a non-negative weak super-solution, then the previous inequality holds for
$q<\frac{1}{2}$.
\end{theorem}
\begin{remark}\label{rlambda1}
Since $C$ in \eqref{caccioppoli} depends only on $M$ and $\|B\|$, then the estimate holds also for
$L^{\l}$ in \eqref{Lr}, uniformly with respect to $\l\in[0,1]$.
\end{remark}
\begin{proof}
Let $u$ be a non-negative weak sub-solution of \eqref{PDE} in $R_r(z_0)$: this means that
\begin{equation}\label{subsolution}
 \int_{R_r(z_0)}-\langle AD  u,D \varphi\rangle-\langle a,D \varphi\rangle
 u+\varphi cu+\varphi Yu\geq 0
\end{equation}
for any non-negative $\varphi\in H_0^1(R_r(z_0))$. We choose $q>\frac{1}{2}$ and assume that
$u^q\in L^{2}(R_r(z_0))$.  Let $\varphi=2qu^{2q-1}\psi^2$ in \eqref{subsolution}, where $\psi\in
C_0^{\infty}(R_r(z_0))$; then we have
\begin{align}
 \langle AD  u,D \varphi\rangle&=\langle AD  u,D (2qu^{2q-1}\psi^2)\rangle\\
 &=2q(2q-1)\psi^2u^{2q-2}\langle AD  u,D  u\rangle+4q\psi u^{2q-1}\langle AD
 u,D \psi\rangle\\ &=\frac{2(2q-1)}{q}\psi^2\langle AD  u^q,D  u^q\rangle+4\psi
 u^q\langle AD  u^q,D \psi\rangle
\intertext{and}
 \int_{R_r(z_0)}\varphi Yu&=\int_{R_r(z_0)}2qu^{2q-1}\psi^2Yu\\ &=\int_{R_r(z_0)}\psi^2Y(u^{2q})\\
 &=\int_{R_r(z_0)}Y(u^{2q}\psi^2)-2\psi u^{2q}Y\psi\\ &=\int_{R_r(z_0)}-\text{\rm tr}(B)u^{2q}\psi^2-2\psi
 u^{2q}Y\psi,
\end{align}
where in the last equality we used the divergence theorem. Moreover,
\begin{align}
 \langle a,D \varphi\rangle u&=\langle a,D (2qu^{2q-1}\psi^2)\rangle u\\ &=
 2q(2q-1)\psi^2u^{2q-1}\langle a,D  u\rangle+4qu^{2q}\psi\langle a,D \psi\rangle\\ &=
 2(2q-1)\psi^2u^q\langle a,D  u^q\rangle+4qu^{2q}\psi\langle a,D \psi\rangle.
\end{align}
Therefore, inequality \eqref{subsolution} can be rewritten as
\begin{align}
 \int_{R_r(z_0)}\frac{2(2q-1)}{q}\psi^2\langle AD  u^q,D  u^q\rangle \leq&\ \int_{R_r(z_0)}
 4u^q|\psi||\langle AD  u^q,D \psi\rangle|+2(2q-1)\psi^2u^q|\langle a,D  u^q\rangle|\\
 &+\int_{R_r(z_0)}u^{2q}h(\psi),
\end{align}
where
\begin{equation}\label{def h}
 h(\psi):=|\text{\rm tr}B|\psi^2+2|\psi||Y\psi|+4q|\psi||\langle a,D \psi\rangle|+2q\psi^2|c|.
\end{equation}
Observe that for any $\e,\delta>0$ we have
\begin{align}
 4u^q|\psi| |\langle AD  u^q,D \psi\rangle|\leq 2\e\psi^2\langle AD
 u^q,D  u^q\rangle+\frac{2}{\e}u^{2q}\langle AD \psi,D\psi\rangle
\end{align}
and
\begin{align}
 2(2q-1)\psi^2u^q|\langle a,D  u^q\rangle|&\leq2(2q-1)\psi^2u^q|a| |D _{m_0}u^q|\leq\frac{2q-1}{\delta}\psi^2u^{2q}|a|^2+\delta(2q-1)\psi^2|D _{m_0}u^q|^2.
\end{align}
These estimates give
\begin{align}
 \int_{R_r(z_0)}2\left(\frac{2q-1}{q}-\e\right)\psi^2\langle AD  u^q,D
 u^q\rangle \leq&\ \int_{R_r(z_0)} \delta(2q-1)\psi^2|D _{m_0}u^q|^2\\
 &+\int_{R_r(z_0)}u^{2q}\left(\frac{2}{\e}\langle
 AD \psi,D \psi\rangle+\frac{2q-1}{\delta}\psi^2|a|^2+h(\psi)\right).
\end{align}
Now, choose $\e=\frac{2q-1}{2q}$, which is positive since $q>\frac{1}{2}$, and use Assumption
\ref{assA} on the matrix $A$ to obtain
\begin{align}
 \frac{2q-1}{q\mu}\int_{R_r(z_0)}\psi^2|D _{m_0}u^q|^2\leq&\ \int_{R_r(z_0)}
 \delta(2q-1)\psi^2|D _{m_0}u^q|^2\\ &+\int_{R_r(z_0)}u^{2q}\left(\frac{4q}{2q-1}\langle
 AD \psi,D \psi\rangle+\frac{2q-1}{\delta}\psi^2|a|^2+h(\psi)\right),
\intertext{or equivalently}
 \left(\frac{2q-1}{q\mu}-\delta(2q-1)\right)\int_{R_r(z_0)}\psi^2|D _{m_0}u^q|^2\leq&\ \int_{R_r(z_0)}u^{2q}\left(\frac{4q}{2q-1}\langle AD \psi,D \psi\rangle+\frac{2q-1}{\delta}\psi^2|a|^2+h(\psi)\right).
\end{align}
The choice $\delta=\frac{1}{2q\mu}$ yields
\begin{equation}\label{last caccioppoli}
 \frac{2q-1}{2q\mu}\int_{R_r(z_0)}\psi^2|D _{m_0}u^q|^2\leq\int_{R_r(z_0)}u^{2q}\Big(\frac{4q}{2q-1}\langle
 AD \psi,D \psi\rangle+2q\mu(2q-1)\psi^2|a|^2+h(\psi)\Big).
\end{equation}
The thesis now follows by making a suitable choice of the function $\psi$. More precisely, if
$z_0=(t_0,x_0)$ we set
\begin{equation}\label{e13}
 \psi(t,x)=\chi\left(\sqrt{|t-t_0|}\right)\chi\left(\|z_0^{-1}\circ(x,0)\|\right),
\end{equation}
where $\chi\in C^{\infty}(\mathbb{R},[0,1])$ is such that
\begin{align}
 \chi(s)=
  \begin{cases}
    1 & \text{ if }\quad s\leq\rho, \\
    0 & \text{ if }\quad s\geq r,
  \end{cases}
 \qquad\text{ and }\quad|\chi '|\leq \frac{2}{r-\rho}.
\end{align}
Note that
\begin{equation}\label{e20}
  \left|\partial_{t}\psi\right|, \left|\partial_{x_{j}}\psi\right|\leq
  \frac{C_{1}}{r-\rho},\qquad j=1,\dots,d,
\end{equation}
where $C_{1}$ is a dimensional constant. With such a choice for $\psi$, we can bound $h(\psi)$ in
\eqref{def h} as
\begin{align}
 |h(\psi)|\leq C\frac{q}{r-\rho}
\end{align}
with $C=C(\cost,\|B\|)$ and inequality \eqref{last caccioppoli} becomes
\begin{align}
 \frac{2q-1}{2q\mu}\int_{R_{\rho}(z_0)}|D _{m_0}u^q|^2&\leq\int_{R_r(z_0)}u^{2q}\left(\frac{4q\mu}{2q-1}|D _{m_0}\psi|^2+
 2q\mu(2q-1)\|  a\| _{\infty}^2+|h(\psi)|\right)\\
 &\leq\frac{Cq^2}{(2q-1)(r-\rho)^2}\int_{R_r(z_0)}u^{2q},
\end{align}
which corresponds to \eqref{caccioppoli}.  The statement concerning super-solutions is proved
similarly. By a standard approximation argument, we can suppose that $u$ is positive; the test
function to be used is $\varphi=2|q|u^{2q-1}\psi^2$ and all the previous inequalities are reversed
due to the negativity of $2q-1$.
\end{proof}
We now state and prove the following Sobolev type inequality.
\begin{theorem}[Sobolev type inequality]\label{Sobolev}
Let $L\in\Kol$, $\l\in[0,1]$ and $L^{\l}$ as in \eqref{Lr}. If $u$ is a non-negative weak
sub-solution of $L^{\l}u=0$ in $R_r(z_0)$, then $u\in L^{2\k}_{\text{\rm loc}}(R_{r}(z_0))$ with
$\k=1+\frac{2}{Q}$ and we have
\begin{equation}\label{sobolev}
 \|u\| _{L^{2\k}(R_{\rho}(z_0))}\leq \frac{C}{r-\rho}\left(\|u\|_{L^{2}(R_{r}(z_0))}+\|  D_{m_0}u\| _{L^{2}(R_{r}(z_0))}\right),
\end{equation}
for every $0<\rho<r\le r_{0}$, satisfying $r-\r<1$, with $C$ dependent only on $\cost,B$ and
$r_{0}$. The same statement holds for non-negative super-solutions.
\end{theorem}
\begin{proof}
We only sketch the proof since it follows closely the idea of the proof of Theorem 3.3 in
\cite{PP2004}. First we consider the case $\l=1$: if $u$ is a non-negative sub-solution of
\eqref{PDE} in $R_r(z_0)$, we represent it in terms of the fundamental solution $\Gg$ in
\eqref{e22and}. To this end, we set by $A_0=\frac{1}{2}\s
\s^{*}$ with $\s$ as in \eqref{sigm} 
and consider the cut-off function $\psi$ introduced in \eqref{e13}. Then, for every $z\in
R_{\rho}(z_0)$ we write
\begin{align}
 u(z)&=(u\psi)(z)\\ &= \int\limits_{R_{r}(z_0)} \left(\langle A_0 D  (u\psi),D
 \Gg(z;\cdot)\rangle-\Gg(z;\cdot)Y(u\psi)\right)(\zeta)d\zeta\\
 &=I_{1}(z)+I_{2}(z)+I_{3}(z)+I_4(z),
\end{align}
where
\begin{align*}
  I_{1}(z) = & \int\limits_{R_{r}(z_0)} \left( \langle A_0 D  \psi,D
  \Gg(z;\cdot) \rangle u \right) (\zeta) d\zeta - \int\limits_{R_{r}(z_0)}
  \left( \Gg(z;\cdot)u Y\psi  \right)(\zeta)d\zeta,\\
  I_{2}(z)= & \int\limits_{R_{r}(z_0)} \left( \langle \left(A_0-A\right) D  u,D
  \Gg(z;\cdot) \rangle \psi \right)(\zeta)d\zeta
  - \int\limits_{R_{r}(z_0)} \left(\Gg(z;\cdot)\langle A D  u,D  \psi
  \rangle\right)(\zeta)d\zeta, \\
  I_{3}(z) = & \int\limits_{R_{r}(z_0)} \left(\langle A D  u,D  (
  \Gg(z;\cdot)\psi) \rangle-\Gg(z;\cdot)\psi Y v+\langle a,D (\Gg(z;\cdot)\psi)\rangle u-\Gg(z;\cdot)\psi c u \right)(\zeta)d\zeta.\\
  I_{4}(z) = & \int\limits_{R_{r}(z_0)} \left(-\langle a,D (\Gg(z;\cdot)\psi)\rangle u+\Gg(z;\cdot)\psi c u \right)(\zeta)d\zeta.
\end{align*}
Since $u$ is a weak sub-solution of \eqref{PDE}, it follows that $I_{3}\leq 0$ and therefore
\begin{align}
 0\leq u\leq I_{1} +I_{2}+ I_4 \qquad \text{a.e. in } R_{\rho}(z_0).
\end{align}
The integrals $I_1$ and $I_2$ can be estimated as in the proof of Theorem 3.3 in \cite{PP2004}.
For the last term we have
\begin{align}
 I_4(z)&= \int\limits_{R_{r}(z_0)} \left(-\langle a,D (\Gg(z;\cdot)\psi)\rangle
 u+\Gg(z;\cdot)\psi c u \right)(\zeta)d\zeta\\ &=\int\limits_{R_{r}(z_0)}
 \left(-\psi\langle a,D \Gg(z;\cdot)\rangle u-\Gg(z;\cdot)\langle
 a,D \psi\rangle u+\Gg(z;\cdot)\psi c u \right)(\zeta)d\zeta\\
 &=\int\limits_{R_{r}(z_0)} \left(-\psi\langle a,D \Gg(z;\cdot)\rangle
 u-\Gg(z;\cdot)u(\langle a,D \psi\rangle -\psi c)  \right)(\zeta)d\zeta\\
 &=\int\limits_{R_{r}(z_0)} \left(-\psi\langle a,D \Gg(z;\cdot)\rangle
 u\right)(\zeta)d\zeta-\int\limits_{R_{r}(z_0)} \left(\Gg(z;\cdot)u(\langle
 a,D \psi\rangle -\psi c) \right)(\zeta)d\zeta\\ &=:I'_4(z)+I''_4(z).
\end{align}
By the potential estimates in \cite{CPP2008}, Theorem 2, and the H\"older inequality, we get
\begin{align}
 \|I_4\|_{L^{2\k}(R_{\rho}(z_0))}&\leq\| I'_4\|_{L^{2\k}(R_{\rho}(z_0))}+\|
 I''_4\|_{L^{2\k}(R_{\rho}(z_0))}\\ &\leq\| I'_4\|_{L^{2\k}(R_{\rho}(z_0))}+C_2\|
 I''_4\|_{L^{2\tilde{\kappa}}(R_{\rho}(z_0))}\\ &\leq C_1 \|
 u\|_{L^{2}(R_{r}(z_0))}+\frac{C_2}{r-\rho}\| u\|_{L^{2}(R_{r}(z_0))},
\end{align}
where $\tilde{\kappa}=1+\frac{4}{Q-2}$, $C_1=C_1(r_{0},B,\|  a \| _{\infty})$ and
$C_2=C_2(r_{0},B,\| a \| _{\infty},\|  c \| _{\infty})$. For the general case when $\l\in[0,1]$,
the proof is completely analogous and be carried out using the fact that the potential estimates
in \cite{CPP2008} are uniform in $\l$.

A similar argument proves the thesis when $u$ is a super-solution. In this case, we introduce the
following auxiliary operator
\begin{align}
 \hat{L}_{0}=\div(A_0 D )+\hat{Y}\quad\text{ where }\quad \hat{Y}:= -\langle x,B
 D \rangle+\partial_{t}.
\end{align}
For any $z=(t,x)$, we set $\hat{z}=(-t,x)$, $v(z)=u(\hat{z})$ and remark that
\begin{equation*}
 D  v(z)= D  u(\hat{z})\quad\text{ and }\quad \hat{Y} v(z)=-Y u(\hat{z})
\end{equation*}
almost everywhere.
Then, if $R$ is a domain which is symmetric with respect to the time variable $t$, since $u$ is a
super-solution, we have
\begin{align}
 &\int\limits_{R}\left(- \langle A (\hat{z}) D  v,D  \phi \rangle -\langle
 a(\hat{z}),D \phi\rangle v+c(\hat{z})\phi v- \phi \hat{Y} v\right)(z)d z\\
 &= \int\limits_{R} \left(-\langle
 A(\hat{z}) D  u(\hat{z}),D  \phi(z)\rangle-\langle a(\hat{z}),D \phi(z)\rangle u(\hat{z})+c(\hat{z})\phi(z) u(\hat{z})
 +\phi(z) Y u(\hat{z})\right) d z \leq 0,
\end{align}
for every non-negative $\phi\in C_{0}^{\infty}(R)$. Then, we represent $v$ in terms of the
fundamental solution $\hat{\Gamma}_{0}$ of $\hat{L}_{0}$ and the proof proceeds as before.
\end{proof}
We are now ready to prove Theorem \ref{moser}.
\begin{proof}[Proof of Theorem \ref{moser}]
As in the proof of Theorem 1 in \cite{CPP2008}, the argument is based on the Moser's iteration method. 
The inequality to be iterated is obtained through a combination of Theorem \ref{Caccioppoli} and
Theorem \ref{Sobolev}, and reads
\begin{equation}\label{e55}
 \|u^{q}\|_{L^{2\kappa}(R_{\rho}(z_0))}\leq \frac{C(\cost,r_{0},q)\sqrt{|q|}}{(r-\r)^{2}}\|u^{q}\|_{L^{2}(R_{r}(z_0))},
\end{equation}
where $0<\rho<r\le r_{0}$ with $r-\rho<1$, $q\neq\frac{1}{2}$ and $u$ is a non-negative weak
solution of $L^{\l}u=0$. From Theorem \ref{Caccioppoli} we see that $C(\cost,r_{0},q)$, as a
function of $q$, is bounded at infinity and diverges at $q=\frac{1}{2}$: this feature is in common
with the equation studied in \cite{CPP2008}. However, the presence of the new factor $\sqrt{|q|}$
in the right hand side of \eqref{e55} requires additional care in the application of the Moser's
iterative procedure. First of all, we fix a sequence of radii
$\rho_n=\Big(1-\frac{1}{2^n}\Big)\rho+\frac{1}{2^n}r$, a sequence of exponents
$q_n=\frac{p}{2}\kappa^n$ and a safety distance, say $\delta$, from $\frac{1}{2}$. The exponent
$p$ is chosen to guarantee that the distance of the resulting exponent $q_n$ from $\frac{1}{2}$ is
at least $\delta$, for each $n\geq 1$. We then iterate inequality \eqref{e55} to obtain
\begin{align}
\|u^{\frac{p}{2}}\|_{L^{\infty}(R_{\rho}(z_0))}\leq f(r-\r)\|u^{\frac{p}{2}}\|_{L^{2}(R_{r}(z_0))}
\end{align}
where, for some $\tilde{C}=\tilde{C}(\cost,r_{0},\d)$, 
\begin{align}
 f(r-\r)&=\prod_{j=0}^{\infty}
 \left(\frac{\tilde{C}\sqrt{|p|}\kappa^{\frac{j}{2}}}{(\r_{j}-\r_{j+1})^{2}}\right)^{\frac{1}{\kappa^{j}}}=\frac{C_1(\cost,r_{0},p)}{(r-\r)^{\frac{Q+2}{2}}},
\end{align}
This proves inequality (\ref{e50}) for $p$ satisfying $|\frac{p}{2}k^n-\frac{1}{2}|\geq\delta$.
The previous restriction is easily relaxed using the monotonicity of the $L^p$-means (see
\cite{CPP2008} for details).
\end{proof}
\begin{remark} The previous proof can be slightly modified to see that Moser's estimate
\eqref{e50} still holds true on the cylinder $R^+_r(z_0):= z_0 \circ \d_{r}
\left(R^{+}_{1}\right)$ with
 $ R^{+}_{1} = \{(t,x)\in\R\times\R^{d}\mid  0<t<1,\, |x|<1\}.$
\end{remark}


\section{Gaussian upper bound}
In this section we prove a Gaussian upper bound for the fundamental solution $\G$ of $L\in\Kol$.
We begin with an important implication of the Moser's estimate \eqref{e50}.
\begin{theorem}[Nash upper bound]\label{t4}
Let $\G$ be the fundamental solution of $L\in\Kol$. Then, there exists a positive constant
$C=C(\cost,T_{0})$ such that
\begin{equation}\label{e30}
 \G(t,x;T,y)\le \frac{C}{(T-t)^{\frac{Q}{2}}},\qquad 0<T-t\le T_{0},\ x,y\in\R^{d}.
\end{equation}
\end{theorem}
\begin{proof}
By Theorem \ref{moser}, with $\r=\frac{1}{2}\sqrt{\frac{T-t}{\max\{T_{0},1\}}}$ and
$r=\sqrt{2}\r$, we have
\begin{align}
 \G(t,x;T,y)&\le \sup_{R_{\r}(t,x)}\G(\cdot,\cdot;T,y)\\
 &\le \frac{C}{(T-t)^{\frac{Q+2}{2}}}\iint\limits_{R_{r}(t,x)}\G(s,\x;T,y)d \x d s\\
 &\le \frac{C}{(T-t)^{\frac{Q+2}{2}}}\int\limits_{t-\frac{T-t}{2\max\{T_{0},1\}}}^{t+\frac{T-t}{2\max\{T_{0},1\}}}\int\limits_{\R^{d}}\G(s,\x;T,y)d \x d s
\intertext{{(since $\int\limits_{\R^{d}}\G(s,\x;T,y)d \x d s\le e^{(T-s)\|c\|_{\infty}}$)}}
 &\le \frac{C}{(T-t)^{\frac{Q}{2}}}.
\end{align}
\end{proof}

An immediate consequence of Theorem \ref{t4} is the following
\begin{corollary}\label{c1bis} \
There exists a positive constant $C=C(\cost,T_{0})$ such that
\begin{align}
 \int\limits_{\R^d}\G^{2}(t,x;T,y) dy \le \frac{C}{(T-t)^{\frac{Q}{2}}},\qquad 0<T-t\le T_{0},\ x\in\R^{d},
\end{align}
and
\begin{align}
 \int\limits_{\R^d}\G^{2}(t,x;T,y) d x \le \frac{C}{(T-t)^{\frac{Q}{2}}},\qquad 0<T-t\le T_{0},\ y\in\R^{d}.
\end{align}
\end{corollary}

Our proof of a Gaussian upper bound for the fundamental solution is adapted to Aronson's method
\cite{Aronson}. The next theorem is a crucial step in this direction.
\begin{theorem}\label{t2}
Fix $y\in \R^{d}$, $\s>0$ and let $u_{0}\in L^{2}(\R^{d})$ be such that $u_{0}(x)=0$ for
$|x-y|<\s$. Let $L\in\Kol$ and suppose that $u$ is a bounded solution to \eqref{PDE} in
$[\y-\s^2,\y[\,\times \R^{d}$ with terminal value $u(\eta,x)=u_{0}(x)$. Then, there exist positive
constants $k$ and $C$ such that for any $\t$ which satisfies
$\eta-\frac{1\wedge\sigma^2}{k}\leq\t\leq\eta$ we have
\begin{equation}\label{e16}
|u((0,e^{-\y B}y)\circ(\t,0))|\le C(\eta-\t)^{-\frac{Q}{4}}
\exp\left(-\frac{\s^{2}}{C(\eta-\t)}\right)\|u_{0}\|_{L^{2}(\R^{d})}.
\end{equation}
The constants $k$ and $C$ depend only on $\cost$. 
\end{theorem}

\begin{proof}
We first prove the thesis for $y=0$. We fix $s$ such that $0\leq\eta-s\leq1\wedge\sigma^2$ and we
define
\begin{align}
 h(t,x)=-\frac{|x|^2}{2(\eta-s)-k(\eta-t)}+\alpha(\eta-t),\qquad \eta-\frac{\eta-s}{k}\leq t\leq \eta,\ x\in\mathbb{R}^d,
\end{align}
with $\alpha$ and $k$ being positive constants to be fixed later on. Moreover, for $R\ge 2$, we
consider  a function $\g_{R}\in C_{0}^{\infty}(\R^{d},[0,1])$ such that $\g_{R}(x)\equiv 1$ for
$|x|\le R-1$, $\g_{R}(x)\equiv 0$ for $|x|\ge R$ with $|D  \g_{R}|$ bounded by a constant
independent of $R$. Then, we multiply both sides of \eqref{PDE} by $\g^2_{R}e^{2h}u$ and we
integrate over $[\tau,\eta]\times \R^{d}$, with $\eta-\frac{\eta-s}{k}\leq\tau\le\eta$, to get
\begin{equation}\label{e19}
\begin{split}
  &\int\limits_{\R^{d}}\g_{R}^{2}e^{2h}u^{2}\vert_{t=\t}\,d x
  -2\iint\limits_{[\tau,\eta]\times\R^{d}} e^{2h}u^{2}
  \left(3\langle A D _{m_{0}}h,D _{m_{0}}h\rangle-Y
  h-2\langle a,D _{m_{0}} h\rangle+\Lambda\right)d x d t\le \\
  &\int\limits_{\R^{d}}\g_{R}^{2}e^{2h}u^{2}\vert_{t=\eta}\,d x+
  2\iint\limits_{[\tau,\eta]\times\R^{d}}e^{2h}u^{2}\left(3\mu
  \left|D _{m_{0}}\g_{R}\right|^{2}+\left|Y\g_{R}^{2}\right|-{2\langle a,D _{m_{0}}\g_{R}\rangle\g_{R}}
  \right)d x d t,
\end{split}
\end{equation}
where $\Lambda$ is a positive constant depending on $\cost$. 
The proof of \eqref{e19} is tedious but routine: all the details are reported in
Appendix \ref{aronapp}.

Next we let $R$ go to infinity in \eqref{e19}: since $u$ is bounded by assumption and
$e^{2h(t,x)}\le e^{-\frac{|x|^{2}}{\eta-s}+2\alpha(\eta-s)}$, the last integral tends to zero and
we get
\begin{equation}\label{e195}
 \int\limits_{\R^{d}} e^{2h}u^{2}\vert_{t=\t}\,d x
  -2\iint\limits_{[\tau,\eta]\times\R^{d}} e^{2h}u^{2}
  \left(3\langle A D _{m_{0}}h,D _{m_{0}}h\rangle-Y
  h-2\langle a,D _{m_{0}} h\rangle+\Lambda\right)d x d t\\
  \le \int\limits_{\R^{d}} e^{2h}u^{2} \vert_{t=\eta}\,d x.
\end{equation}
\noindent We now claim that, by a suitable choice of $k$ and $\alpha$, only dependent on
$\cost,B$,
we have
\begin{equation}\label{e17}
  3\langle A D _{m_{0}}h,D _{m_{0}}h\rangle-Y
  h-2\langle a,D _{m_{0}} h\rangle+\Lambda\le 0,\qquad \eta-\frac{\eta-s}{k}\le t\le\eta,\ x\in\R^{d}.
\end{equation}
Indeed, letting $\d=2(\eta-s)-k(\eta-t)$ to ease the notation, we have
\begin{align}
 3\langle A D _{m_{0}}h,D _{m_{0}}h\rangle-Yh-2\langle a,D _{m_{0}}
 h\rangle+\Lambda
 &\leq \frac{12\mu|x|^2}{\d^2}+\frac{2\left\|B\right\||x|^2}{\d}-\frac{k|x|^2}{\d^2}-\alpha+\frac{4\langle a,x\rangle}{\d}+\Lambda\\
 &{\le \frac{|x|^2}{\d^2}(12\mu+2\d\left\|B\right\|-k+2)-\alpha+2\left\|a\right\|^{2}_{\infty}+\Lambda}\\
 &{\leq \frac{|x|^2}{\d^2}(12\mu+4\left\| B\right\|-k+2)-\alpha+2\left\|a\right\|^{2}_{\infty}+\Lambda,}
\end{align}
and, with $\alpha=2\left\|a\right\|^{2}_{\infty}+\Lambda$ and $k$ big enough, the last term can be
made negative.

From \eqref{e17} and \eqref{e195}, we derive the inequalities
\begin{align}
  \max_{t\in\, ]\eta-\frac{\eta-s}{k},\eta[}\ \int\limits_{\left|D\left(\frac{2\sqrt{k}}{\sqrt{\eta-s}}\right)
  x\right|\le 1}e^{2h(t,x)}u^{2}(t,x)d x
  &\le \max_{t\in\, ]\eta-\frac{\eta-s}{k},\eta[}\ \int\limits_{\R^{d}} e^{2h(t,x)}u^{2}(t,x)d x\\
  \label{est3}
  &\le \int\limits_{|x|\ge\s}e^{2h(\eta,x)}u_{0}^{2}(x)d x.
\end{align}
Now we notice that, by definition, for every $t\in\,]\eta-\frac{\eta-s}{k},\eta]$ we have
\begin{align}
 2h(t,x)&\ge -\frac{2|x|^{2}}{\y-s}=
\intertext{(setting $\d=\frac{\sqrt{\eta-s}}{2\sqrt{k}}$)}
 &= -\frac{2\left|\Dil(\d)\Dil(\d^{-1})x\right|^{2}}{\y-s}\ge
\intertext{(if
$\left|D\left(\frac{2\sqrt{k}}{\sqrt{\eta-s}}\right)x\right|=\left|\Dil\left(\d^{-1}\right)x\right|\le
1$)}
 &\ge -\frac{2\left\|\Dil(\d)\right\|^{2}}{\y-s}\ge
\intertext{(since $\d<1$ by assumption)}\label{est1}
 &\ge -\frac{2\d^{2}}{\y-s}= -\frac{1}{2k}.
\end{align}
On the other hand, if $|x|\ge\s$, we have
\begin{align}\label{est2}
 -2h(\eta,x)=\frac{2|x|^2}{2(\eta-s)}\geq\frac{\sigma^2}{\eta-s}.
\end{align}
Plugging estimates \eqref{est1} and \eqref{est2} into \eqref{est3}, we get
\begin{equation}\label{e23}
  \max_{t\in ]\eta-\frac{\eta-s}{k},\eta[}\ \int\limits_{\left|D\left(\frac{2\sqrt{k}}{\sqrt{\eta-s}}
  \right)x\right|\le 1}u^{2}(t,x)d x \le e^{\frac{1}{2k}}
  \exp\left(-\frac{\s^{2}}{\eta-s}\right)\|u_{0}\|^{2}_{L^{2}(\R^{d})}.
\end{equation}
Finally, we rely on Theorem \ref{moser} in order to get the desired estimate \eqref{e16}. We let
$\tau=\eta-\frac{\eta-s}{k}$ and we observe that $\tau\in [\eta-\frac{1}{k},\eta]$ and
$\eta-s=k(\eta-\t)$: thus we have
\begin{align}
 |u(\t,0)|^{2}&\le \sup_{R^{+}_{\frac{\sqrt{\eta-s}}{4\sqrt{k}}}(\t,0)}|u|^{2}\le
\intertext{(by \eqref{e50})}
 &\le \frac{C}{(\eta-s)^{\frac{Q+2}{2}}}\iint\limits_{R^{+}_{\frac{\sqrt{\eta-s}}{2\sqrt{k}}}(\t,0)}u^{2}(t,x)d x d t\\
 &=  \frac{C}{(\eta-s)^{\frac{Q+2}{2}}} \int\limits^{\t+\frac{\eta-s}{4k}}_{\t}
 \int\limits_{\left|D\left(\frac{2\sqrt{k}}{\sqrt{\eta-s}}\right)x\right|\le 1} u^{2}(t,x)d x d t
\intertext{(by \eqref{e23})}
 &\le  \frac{C}{(\eta-s)^{\frac{Q}{2}}}\exp\left(-\frac{\s^{2}}{C(\eta-s)}\right) \|u_{0}\|^{2}_{L^{2}(\R^{d})}\\
 &{= \frac{C}{k^{\frac{Q}{2}}(\eta-\t)^{\frac{Q}{2}}}\exp\left(-\frac{\s^{2}}{Ck(\eta-\t)}\right)
 \|u_{0}\|^{2}_{L^{2}(\R^{d})},}
\end{align}
where the constant $C=C(\cost,k)$. This yields \eqref{e16} in the case $y=0$.

For the general case, for any fixed $u$, $u_{0}$ and $(\y,y)$ as in the statement, we set
\begin{align}
 v(\t,x)= u\left((0,e^{-\y B}y)\circ(\t,x)\right)\qquad \t<\eta, \ x\in\R^{d},
\end{align}
{and observe that $v(\eta,x)=u(\eta,y+x)=u_{0}(y+x)=0$ for $|x|\le\s$. Moreover, by the invariance
property of the vector field $Y$ with respect to the left translation $\ell_z$, we have
 $L^{(z)} v=0,$ 
where {$L^{(z)}:=L\circ \ell_{z}\in\Kol$} and $z=(0,e^{-\y B}y)$.} Thus, we get as before
\begin{align}
 |u(z\circ(\t,0))|&= |v(\t,0)|\leq\frac{C}{(\eta-\t)^{\frac{Q}{4}}}\exp\left(-\frac{\s^{2}}{C(\eta-\t)}\right)
\|u_{0}\|_{L^{2}(\R^{d})}
\end{align}
completing the proof.
\end{proof}

\noindent The following corollary is a simple consequence of Theorems \ref{t2}.
\begin{corollary}\label{c1}
There exists two positive constants $k$ and $C$, that depend only on $\cost$, such that for every
$\sigma>0$ and $\y\in\R$, we have
\begin{equation}\label{e24}
 \int\limits_{\left|\x-e^{(\y-t)B}x\right|\ge\s}\G^{2}(t,x;\eta,\xi)d\x\le
 \frac{Ce^{-\frac{\s^{2}}{C(\eta-t)}}}{(\eta-t)^{\frac{Q}{2}}},\qquad (t,x)\in\Big[\eta-\frac{1\wedge\sigma^2}{k},\eta\Big[\times\R^{d},
\end{equation}
and
\begin{equation}\label{e24dual}
 \int\limits_{\left|x-e^{(t-\y)B}\x\right|\ge\s}\G^{2}(t,x;\eta,\xi)dx\le
 \frac{Ce^{-\frac{\s^{2}}{C(\eta-t)}}}{(\eta-t)^{\frac{Q}{2}}},\qquad (t,x)\in\Big[\eta-\frac{1\wedge\sigma^2}{k},\eta\Big[\times\R^{d}.
\end{equation}
\end{corollary}
\begin{proof}
First of all we observe that
\begin{align}
 \int\limits_{\left|\x-e^{(\y-t)B}x\right|\ge\s}\G^{2}(t,x;\eta,\xi)d\x&=
 \int\limits_{\left|\x-y\right|\ge \s} \G^{2}\left(t,e^{(t-\eta)B}y;\eta,\xi\right)d\x\\
 &=\int\limits_{\left|\x-y\right|\ge \s} \G^{2}\left((0,e^{-\y B}y)\circ(t,0);\eta,\xi\right)d\x.
\end{align}
Now the function
\begin{align}
 u(s,w): =\int\limits_{\left|\x-y\right|\ge \s} \G(s,w;\eta,\xi)\G((0,e^{-\y
 B}y)\circ(t,0);\eta,\xi)d\x,
\end{align}
is a non-negative solution to \eqref{PDE} for $s<\y$, with terminal condition
  $$u(\eta,w)=
  \begin{cases}
    0 & \text{ if } |w-y|<\s, \\
    \G((0,e^{-\y B}y)\circ(t,0);\eta,w) & \text{ if } |w-y|\ge \s.
  \end{cases}
  $$
Setting $(s,w)=(0,e^{-\y B}y)\circ(t,0)$, from Theorem \ref{t2} we infer
\begin{align}
 \int\limits_{\left|\x-y\right|\ge \s}\G^{2}\left((0,e^{-\y B}y)\circ(t,0);\eta,\xi\right)d\x&= u((0,e^{-\y B}y)\circ(t,0))\\ &\le
 \frac{Ce^{-\frac{\s^{2}}{C(\eta-t)}}}{(\eta-t)^{\frac{Q}{4}}}\|\G\left((0,e^{-\y
 B}y)\circ(t,0),\eta,\cdot\right)\|_{L^{2}(\R^{d})}.
\end{align}
Then the thesis follows directly from Corollary \ref{c1bis}. Inequality \eqref{e24dual} is proved
similarly.
\end{proof}

We are now in position to prove our main result.
\begin{proof}[Proof of Theorem \ref{t1}]
The proof proceeds in a series of steps.

\noindent {\bf Step 1.} We first prove the thesis for $y=0$ and $T-t=\frac{1}{k}$, with $k$ as in
Theorem \ref{t2}. We fix $x\in\R^{d}$ and set
\begin{equation}\label{ee1}
 \s(x)=\frac{|x|}{2 \|e^{\frac{T-t}{2}B}\|}.
\end{equation}
If $\s(x)\le 1$, that is $|x|\le 2\|e^{\frac{T-t}{2}B}\|$, then the thesis is a direct consequence
of Theorem \ref{t4} and the fact that, by assumption, $T-t=\frac{1}{k}$ is fixed with $k$
dependent only upon $\cost$.

On the other hand, if $\s(x)\ge 1$, by the Chapman-Kolmogorov identity and putting
$\eta=T-\frac{T-t}{2}$, we have
\begin{align}
 \G(t,x;T,0)= \int\limits_{\R^{d}}\G(t,x;\eta,\xi)\G(\eta,\xi;T,0)d\x=J_1+J_2,
\end{align}
where
\begin{align}
 J_{1}&:=\ \int\limits_{\big|\x-e^{\frac{T-t}{2}B}x\big|\ge\s(x)}\G(t,x;\eta,\xi)\G(\eta,\xi;T,0)d\x,\\
 J_{2}&:=\ \int\limits_{\big|\x-e^{\frac{T-t}{2}B}x\big|<\s(x)}\G(t,x;\eta,\xi)\G(\eta,\xi;T,0)d\x.
\end{align}
By the Cauchy-Schwarz inequality, we have
\begin{align}
 \left(J_{1}\right)^{2}&\le\int\limits_{\big|\x-e^{\frac{T-t}{2}B}x\big|\ge\s(x)}\G^2(t,x;\eta,\xi)d\x
 \int\limits_{\big|\x-e^{\frac{T-t}{2}B}x\big|\ge\s(x)} \G^2(\eta,\xi;T,0)d\x
\intertext{(by \eqref{e24} and Corollary \ref{c1bis})}
 &\leq \frac{Ce^{-\frac{\s^{2}(x)}{C(T-t)}}}{(T-t)^{Q}}\\
 &= C k^{Q}\exp\left(-\frac{k|x|^2}{4C\|e^{\frac{1}{2k}B}\|^{2}}\right).
\end{align}
In order to estimate $J_{2}$, we first note that if $\big|\x-e^{\frac{T-t}{2}B}x\big|<\s(x)$ then,
recalling also the definition \eqref{ee1} of $\s(x)$, we have
\begin{align}\label{ee2}
 |\x|\ge\big|e^{-\frac{T-t}{2}}x\big|-\big|\x-e^{-\frac{T-t}{2}}x\big|\ge \frac{|x|}{\|e^{\frac{T-t}{2}B}\|}-\s(x)=\s(x).
\end{align}
Thus, by \eqref{ee2} and using again the Cauchy-Schwarz inequality, we have
\begin{align}
 \left(J_{2}\right)^{2}&\le 
 \int\limits_{\left|\x\right|\geq\s(x)} \G^2(\eta,\xi;T,0)d\x \int\limits_{\left|\x\right|\geq\s(x)}\G^2(t,x;\eta,\xi)d\x
\intertext{(by \eqref{e24dual})}
 &\leq \frac{Ce^{-\frac{\s^{2}(x)}{C(T-t)}}}{(T-t)^{\frac{Q}{2}}} \int\limits_{\mathbb{R}^d}\G^2(t,x;\eta,\xi)d\x
\intertext{(by Corollary \ref{c1bis})}
 &\leq \frac{C}{(T-t)^{Q}}e^{-\frac{\s^{2}(x)}{C(T-t)}}\\
 &= C k^{Q}\exp\left(-\frac{k|x|^2}{4C\|e^{\frac{1}{2k}B}\|^{2}}\right).
\end{align}
This completes the proof of the case $\s(x)\geq 1$. In conclusion we have proved estimate
\eqref{thes} for $T-t=\frac{1}{k}$, that is
\begin{equation}\label{last}
 \Gamma(t,x;T,0)\leq Ce^{-\frac{|x|^2}{C}},\qquad T-t=\frac{1}{k}, \ x\in\R^{d},
\end{equation}
with the constant $C$ only dependent on $\cost$ and $B$. 
Actually, the same estimate holds also for the fundamental solution $\G^{\l}$ of $L^{\l}$ in
\eqref{Lr}, {\it with $C$ independent of $\l\in[0,1]$:} in fact, all the results of this section
derive from the Moser's estimate, Theorem \ref{moser}, which is uniform in $\l\in[0,1]$.

\medskip\noindent {\bf Step 2.} We use a scaling argument to generalize estimate \eqref{last} to the case $0<
T-t\leq\frac{1}{k}$; precisely, we prove that
\begin{equation}\label{last1}
 \Gamma(t,x;T,0)\leq\frac{C}{(T-t)^{\frac{Q}{2}}}e^{-\frac{|x|^2}{C(T-t)}},\qquad 0< T-t\leq\frac{1}{k},\
 x\in\R^{d}.
\end{equation}
For $\l\in[0,1]$, we set
 $$\Gamma^{\l}(t,x;T,0)=\l^{Q}\Gamma(\d_{\l}(t,x);\d_{\l}(T,0))$$
and observe that, since the Jacobian $J \Dil(\l)$ equals $\l^{Q}$, we have that $\Gamma^{\l}$ is a
fundamental solution of the operator $L^{(\l)}$ in \eqref{Lr}.

Now, fix $t$ such that $0<T-t\leq\frac{1}{k}$ and set $\l=k(T-t)$. Then we have
\begin{align}
 \Gamma(t,x;T,0)&=\l^{-\frac{Q}{2}}\Gamma^{(\sqrt{\l})}\left(\frac{t}{\l},\Dil\left(\frac{1}{\sqrt{\l}}\right)x;\frac{T}{\l},0\right)\le
\intertext{(by \eqref{last})}
 &\le C\l^{-\frac{Q}{2}}e^{-\frac{1}{C}\left|\Dil\left(\frac{1}{\sqrt{\l}}\right)x\right|^{2}}
\end{align}
which proves \eqref{last1}.

\medskip\noindent {\bf Step 3.} We now remove the condition $y=0$. Let 
$z=(0,e^{-TB}y)$ and $\Gamma^{(z)}$ be the fundamental solution of the operator $L^{(z)}:=L\circ
\ell_z$. 
{Since $L^{(z)}\in\Kol$, we have that $\Gamma^{(z)}$ satisfies the estimate \eqref{last1} and
hence we} obtain
\begin{align}
\Gamma(t,x;T,y)&= 
 \Gamma^{(z)}(z^{-1}\circ (t,x);T,0)\\ &= \Gamma^{(z)}(t,x-e^{-(T-t)B}y;T,0)\\
 &\leq \frac{C}{(T-t)^{\frac{Q}{2}}}\exp\left(-\frac{1}{C}\left|\Dil\left(\frac{1}{\sqrt{T-t}}\right)\left(x-e^{-(T-t)B}y\right)\right|^{2}\right),\qquad 0< T-t\leq\frac{1}{k},\ x,y\in\R^{d}.
\end{align}

\medskip\noindent {\bf Step 4.} In the last step we relax the restriction on the length of the time interval. We first
suppose that $0< T-t\leq\frac{2}{k}$ and set $\tau=\frac{T-t}{2}$. By the Chapman-Kolmogorov
identity we have
\begin{align}
 \Gamma(t,x;T,y)&= \int_{\mathbb{R}^d}\Gamma(t,x;t+\tau,\xi)\Gamma(t+\tau,\xi;T,y)d\xi\\
 &\leq \frac{C}{\tau^{Q}}\int_{\mathbb{R}^d}e^{-\frac{1}{C}\left|\Dil\left(\frac{1}{\sqrt{\tau}}\right)\left(x-e^{-\t B}\xi\right)\right|^2}
 e^{-\frac{1}{C}\left|\Dil\left(\frac{1}{\sqrt{\tau}}\right)\left(\x-e^{-\t B}y\right)\right|^2}d\xi\\
 &\leq \frac{C}{\tau^{Q}}\int_{\mathbb{R}^d}e^{-\frac{1}{C}\left|\Dil\left(\frac{1}{\sqrt{\tau}}\right)\left(x-e^{-\t B}\xi\right)\right|^2}
 e^{-\frac{1}{C}\left|\Dil\left(\frac{1}{\sqrt{\tau}}\right)\left(e^{-\t B}\xi-e^{-(T-t)B}y\right)\right|^2}d\xi
\intertext{(by the Chapman-Kolmogorov identity for a standard Gaussian kernel)}
 &\leq \frac{C}{(T-t)^{\frac{Q}{2}}}e^{-\frac{1}{C}\left|\Dil\left(\frac{1}{\sqrt{T-t}}\right)\left(x-e^{-(T-t)B}y\right)\right|^2}.
\end{align}
Iterating this procedure we can extend the estimate to any bounded time interval and this
concludes the proof.
\end{proof}

\titleformat{\section}{\large\bfseries}{\appendixname~\thesection .}{0.5em}{}
\begin{appendices}

\section{}\label{aronapp}
\noindent The proof of estimate \eqref{e19} is based on standard integration by parts and the
repeated use of the chain rule and the elementary inequality $2|\a\b|\leq \delta
\a^2+\frac{\b^2}{\delta}$. We start multiplying both sides of \eqref{PDE} by $\g^2_{R}e^{2h}u$ and
integrating over $[\tau,\eta]\times \R^{d}$: we get
\begin{align}
 0 
 &=\iint\limits_{[\tau,\eta]\times\R^{d}}\div(AD
 u)\gamma^2_Re^{2h}u+\iint\limits_{[\tau,\eta]\times\R^{d}}(Yu)\gamma^2_Re^{2h}u+\iint\limits_{[\tau,\eta]\times\R^{d}}\div(au)\gamma^2_Re^{2h}u
 +\iint\limits_{[\tau,\eta]\times\R^{d}}cu^2\gamma^2_Re^{2h}\\ \label{appeq}
 &=:\mathbf{I}_1+\mathbf{I}_2+\mathbf{I}_3+\mathbf{I}_4.
\end{align}
Now, we have
\begin{align}
 \mathbf{I}_1=&\, -\iint\limits_{[\tau,\eta]\times\R^{d}}\langle AD  u,D (\gamma^2_Re^{2h}u)\rangle\\
 =&\, -\iint\limits_{[\tau,\eta]\times\R^{d}}2\langle AD  u,D \gamma_R\rangle \gamma_Re^{2h}u+2\langle AD  u,D
 h\rangle\gamma^2_Re^{2h}u+\langle AD  u,D  u\rangle\gamma^2_Re^{2h}\\
 \leq&\ \iint\limits_{[\tau,\eta]\times\R^{d}}\frac{1}{3}\langle AD  u,D
 u\rangle\gamma_R^2e^{2h}+ 3\langle AD \gamma_R,D \gamma_R\rangle u^2e^{2h}\\
  &+\iint\limits_{[\tau,\eta]\times\R^{d}}\frac{1}{3}\langle AD  u,D
 u\rangle\gamma^2_Re^{2h}+3\langle AD  h,D  h\rangle\gamma^2_Re^{2h}u^2-\iint\limits_{[\tau,\eta]\times\R^{d}}\langle AD  u,D  u\rangle\gamma^2_Re^{2h}\\
 =&\, \iint\limits_{[\tau,\eta]\times\R^{d}}-\frac{1}{3}\langle AD  u,D
 u\rangle\gamma_R^2e^{2h}+ 3\langle AD \gamma_R,D \gamma_R\rangle u^2e^{2h}+3\langle
 AD  h,D  h\rangle\gamma^2_Re^{2h}u^2\\
 \leq&\ \iint\limits_{[\tau,\eta]\times\R^{d}}-\frac{1}{3\mu}|D _{m_0} u|^2\gamma_R^2e^{2h}+ 3\langle
 AD \gamma_R,D \gamma_R\rangle u^2e^{2h}+3\langle AD  h,D
 h\rangle\gamma^2_Re^{2h}u^2;
\intertext{moreover, we have}
 \mathbf{I}_2=&\, \frac{1}{2}\iint\limits_{[\tau,\eta]\times\R^{d}}\gamma^2_Re^{2h}Yu^2\\
 =&\, \frac{1}{2}\iint\limits_{[\tau,\eta]\times\R^{d}}\gamma^2_Re^{2h}(\langle Bx,D \rangle+\partial_t)u^2\\
 =&\, -\frac{1}{2}\iint\limits_{[\tau,\eta]\times\R^{d}}u^2\langle Bx,D (\gamma^2_R e^{2h})\rangle+u^2\gamma^2_R e^{2h}\text{\rm tr}B\\
  &+\frac{1}{2}\int_{\R^d}u^2\gamma^2_Re^{2h}|_{\tau}^{\eta}-\iint\limits_{[\tau,\eta]\times\R^{d}}u^2\gamma^2_R
 e^{2h}\partial_th\\
 =&\, -\iint\limits_{[\tau,\eta]\times\R^{d}}u^2\langle Bx,D \gamma_R\rangle\gamma_R e^{2h}+u^2\langle Bx,D  h\rangle\gamma^2_R
 e^{2h}+u^2\gamma^2_R e^{2h}\text{\rm tr}B\\
  &+\frac{1}{2}\int_{\R^d}u^2\gamma^2_Re^{2h}|_{\tau}^{\eta}-\iint\limits_{[\tau,\eta]\times\R^{d}}u^2\gamma^2_R e^{2h}\partial_th\\
 =&\, -\iint\limits_{[\tau,\eta]\times\R^{d}}u^2\langle Bx,D \gamma_R\rangle\gamma_R e^{2h}+u^2\gamma^2_R e^{2h}Yh+u^2\gamma^2_R e^{2h}\text{\rm tr}B
 +\frac{1}{2}\int_{\R^d}u^2\gamma^2_Re^{2h}|_{\tau}^{\eta};
\intertext{finally, we also have}
 \mathbf{I}_3=&\, \iint\limits_{[\tau,\eta]\times\R^{d}}\div(au)\gamma^2_Re^{2h}u\\
 =&\, -\iint\limits_{[\tau,\eta]\times\R^{d}}u\langle a,D (\gamma^2_Re^{2h}u)\rangle\\
 =&\, -\iint\limits_{[\tau,\eta]\times\R^{d}}2u^2\gamma_Re^{2h}\langle
 a,D \gamma_R\rangle+2u^2\gamma^2_Re^{2h}\langle a,D  h\rangle+u\gamma^2_Re^{2h}\langle
 a,D  u\rangle\\
 \leq&\, -\iint\limits_{[\tau,\eta]\times\R^{d}}2u^2\gamma_Re^{2h}\langle
 a,D \gamma_R\rangle+2u^2\gamma^2_Re^{2h}\langle a,D  h\rangle+\iint\limits_{[\tau,\eta]\times\R^{d}}\frac{\delta}{2}|D _{m_0} u|^{2}
 \gamma_R^2e^{2h}+\frac{1}{2\delta}|a|^2u^2\gamma_R^2e^{2h}.
\end{align}
Plugging the above estimates into \eqref{appeq}, we get
\begin{align}
 0=&\, \mathbf{I}_1+\mathbf{I}_2+\mathbf{I}_3+\mathbf{I}_4\\
 \leq&\, \iint\limits_{[\tau,\eta]\times\R^{d}}-\frac{1}{3\mu}|D _{m_0} u|^2\gamma_R^2e^{2h}+ 3\langle
 AD \gamma_R,D \gamma_R\rangle u^2e^{2h}+3\langle AD  h,D
 h\rangle\gamma^2_Re^{2h}u^2\\
 &-\iint\limits_{[\tau,\eta]\times\R^{d}}u^2\langle Bx,D \gamma_R\rangle\gamma_R e^{2h}+u^2\gamma^2_R
 e^{2h}(Yh+\text{\rm tr}B)+\frac{1}{2}\int_{\R^d}u^2\gamma^2_Re^{2h}|_{\tau}^{\eta}\\
 &-\iint\limits_{[\tau,\eta]\times\R^{d}}2u^2\gamma_Re^{2h}\langle a,D \gamma_R\rangle+2u^2\gamma^2_Re^{2h}\langle a,D
 h\rangle+\iint\limits_{[\tau,\eta]\times\R^{d}}\frac{\delta}{2}|D _{m_0} u|^{2} \gamma_R^2e^{2h}+\frac{1}{2\delta}|a|^2u^2\gamma_R^2e^{2h}\\
 &+\iint\limits_{[\tau,\eta]\times\R^{d}}cu^2\gamma^2_Re^{2h},
\end{align}
and by choosing $\delta=\frac{2}{3\mu}$, we obtain
\begin{align}
 \frac{1}{2}\int_{\R^d}u^2\gamma^2_Re^{2h}|_{\tau}\leq&\, \frac{1}{2}\int_{\R^d}u^2\gamma^2_Re^{2h}|_{\eta}+
 \iint\limits_{[\tau,\eta]\times\R^{d}}3\langle AD \gamma_R,D \gamma_R\rangle u^2e^{2h}+3\langle
 AD  h,D  h\rangle\gamma^2_Re^{2h}u^2\\
  &-\iint\limits_{[\tau,\eta]\times\R^{d}}u^2\langle Bx,D \gamma_R\rangle\gamma_R
 e^{2h}+u^2\gamma^2_R e^{2h}(Yh+\text{\rm tr}B)\\
  &-\iint\limits_{[\tau,\eta]\times\R^{d}}2u^2\gamma_Re^{2h}\langle
 a,D \gamma_R\rangle+2u^2\gamma^2_Re^{2h}\langle a,D  h\rangle+\iint\limits_{[\tau,\eta]\times\R^{d}}\frac{3\mu}{4}|a|^2u^2\gamma_R^2e^{2h}\\
  &+\iint\limits_{[\tau,\eta]\times\R^{d}}cu^2\gamma^2_Re^{2h}\\
 =&\, \frac{1}{2}\int_{\R^d}u^2\gamma^2_Re^{2h}|_{\eta}\\
 &+\iint\limits_{[\tau,\eta]\times\R^{d}}(3\langle AD \gamma_R,D \gamma_R\rangle-\langle
 Bx,D \gamma_R\rangle\gamma_R-2\gamma_R\langle a,D \gamma_R\rangle)u^2e^{2h}\\
 &+\iint\limits_{[\tau,\eta]\times\R^{d}}(3\langle AD  h,D  h\rangle-Yh-2\langle a,D
 h\rangle+\frac{3\mu}{4}|a|^2+c-\text{\rm tr}B)u^2\gamma^2_Re^{2h}
\end{align}
Hence, setting $\Lambda=\frac{3\mu}{4}\|a\|_{\infty}^2+\|c\|_{\infty}+|\text{\rm tr}B|$, we obtain
\begin{align}
 \frac{1}{2}\int_{\R^d}u^2\gamma^2_Re^{2h}|_{\tau}\leq&\, \frac{1}{2}\int_{\R^d}u^2\gamma^2_Re^{2h}|_{\eta}\\
  &+\iint\limits_{[\tau,\eta]\times\R^{d}}(3\langle AD \gamma_R,D \gamma_R\rangle-\langle
 Bx,D \gamma_R\rangle\gamma_R-2\gamma_R\langle a,D \gamma_R\rangle)u^2e^{2h}\\
  &+\iint\limits_{[\tau,\eta]\times\R^{d}}(3\langle AD  h,D  h\rangle-Yh-2\langle a,D
 h\rangle+\Lambda)u^2\gamma^2_Re^{2h},
\end{align}
which yields estimate \eqref{e19}.
\end{appendices}

%
%

\bibliographystyle{siam}
\bibliography{BibTeX-Final}

\end{document}